\documentclass[12pt]{amsart}

\usepackage{amssymb,amscd,amsthm, verbatim,amsmath,color,fancyhdr, mathrsfs}
\usepackage{graphicx}
\usepackage{turnstile}

\usepackage[letterpaper, left=2.5cm, right=2.5cm, top=2.5cm,
bottom=2.5cm,dvips]{geometry}

\newtheorem{thm}{Theorem}[section]

\newtheorem{lem}[thm]{Lemma}

\title{Estimates of the bounds of $\pi(x)$ and $\pi((x+1)^{2}) - \pi(x^{2})$}

\author{C. P. Wilson}

\begin{document}
\maketitle

\begin{abstract}
We show the following bounds on the prime counting function $\pi(x)$ using principles from analytic number theory, giving an estimate: 
    $$
    2 \log 2 \geq \limsup _{x \rightarrow \infty} \frac{\pi(x)}{x / \log x} \geq \liminf _{x \rightarrow \infty} \frac{\pi(x)}{x / \log x} \geq \log 2
    $$
\begin{center}
    $\wedge \  \exists $m, $M \in \mathbb{R}^{+}$  such that
\end{center}    
    $$
    M>\frac{\pi(x)}{x / \log x}>m
    $$
    for all $x$ sufficiently large.

We also conjecture about the bounding of $\pi((x+1)^{2}) - \pi(x^{2})$, as is relevant to Legendre's conjecture about the number of primes in the aforementioned interval such that:

$$
\left \lfloor\frac{1}{2}\left(\frac{\left(x+1\right)^{2}}{\log\left(x+1\right)}-\frac{x^{2}}{\log x}\right)-\frac{\left(\log x\right)^{2}}{\log\left(\log x\right)}\right \rfloor \leq \pi((x+1)^{2}) - \pi(x^{2}) \leq
$$
$$
\left \lfloor\frac{1}{2}\left(\frac{\left(x+1\right)^{2}}{\log\left(x+1\right)}-\frac{x^{2}}{\log x}\right) + \log^{2}x\log\log x \right \rfloor
$$
\end{abstract}

\section{Introduction}
Let us define:

$$
\Theta(x):=\sum_{p \leq x} \log p, \  \Psi(x):=\sum_{1 \leq n \leq x} \Lambda(n)=\sum_{\substack{p^{m} \leq x \\ m \geq 1}} \log p,
$$

clearly for prime $p$, and $\Lambda(n)=\left\{\begin{array}{ll}
\log p & \text { if } n=p^{k} \text { for some prime } p \text { and integer } k \geq 1 \\
0 & \text { otherwise. }
\end{array}\right.$

The following are simple statements from real analysis that are required for rigorousness' sake: let $\left\{x_{n}\right\}$ be a sequence of real numbers and $L$ be a real number with the following two properties: $\forall \epsilon>0,$ $\exists N$ such that $x_{n}<L+\epsilon$, $\forall n \geq N$. $\forall \epsilon>0\ \wedge \ N \geq 1, \exists n \geq N$ with $x_{n}>L-\epsilon .$ We thus define $L$ as:
$$
\limsup _{n \rightarrow \infty} x_{n}=L
$$
Thus on the contrary we must have:
$$
\liminf _{n \rightarrow \infty} x_{n}=-\limsup _{n \rightarrow \infty}-x_{n}
$$
\pagebreak

\section{Necessary preliminary results}
\begin{thm}
$$
\liminf _{x \rightarrow \infty} \frac{\pi(x)}{x / \log x}=\liminf _{x \rightarrow \infty} \frac{\Psi(x)}{x}=\liminf _{x \rightarrow \infty} \frac{\Theta(x)}{x}
$$
\begin{center}
    $\therefore$
\end{center}
$$
\limsup _{x \rightarrow \infty} \frac{\pi(x)}{x / \log x}=\limsup _{x \rightarrow \infty} \frac{\Psi(x)}{x}=\limsup _{x \rightarrow \infty} \frac{\Theta(x)}{x}
$$
\end{thm}

\begin{proof}
Clearly $\Theta(x) \leq \Psi(x),$ such that
$$
\limsup _{x \rightarrow \infty} \frac{\Psi(x)}{x} \geq \limsup _{x \rightarrow \infty} \frac{\Theta(x)}{x}
$$
Also, if $p$ is a prime and $p^{m} \leq x<p^{m+1},$ then $\log p$ occurs in the sum for $\Psi(x)$ exactly $m$ times.\cite{Murty}
$\therefore$
$$
\begin{aligned}
\Psi(x)=& \sum_{\substack{p^{m} \leq x \\ p \text { prime } \\
 m \geq 1 }} \log p \\
=& \sum_{p \leq x}\left[\frac{\log x}{\log p}\right] \log p \\
\leq & \sum_{\substack{p \leq x \\ p \text { prime }}} \log x \\
=& \pi(x) \log x
\end{aligned}
$$
$$
\limsup _{x \rightarrow \infty} \frac{\Psi(x)}{x} \leq \limsup _{x \rightarrow \infty} \frac{\pi(x)}{x / \log x}
$$
Now fix $\alpha \in(0,1)$. Given $x>1$,
$$
\Theta(x)=\sum_{p \leq x \atop p \text { prime }} \! \log p \geq \sum_{x^{\alpha}<p \leq x \atop p \text { prime }} \! \log p.
$$
It is clear that all $p$ from the second sum satisfy: $\log p>\alpha \log x .$

$\therefore$
$$
\begin{aligned}
\Theta(x)>& \alpha \log x \sum_{x^{\alpha}<p \leq x \atop p \text { prime }} 1 \\
&=\alpha \log x\left(\pi(x)-\pi\left(x^{\alpha}\right)\right) \\
&>\alpha \log x\left(\pi(x)-x^{\alpha}\right)
\end{aligned}
$$

$\ni$
$$
\frac{\Theta(x)}{x}>\frac{\alpha \pi(x)}{x / \log x}-\frac{\alpha \log x}{x^{1-\alpha}}
$$

\pagebreak
$\forall \alpha \in(0,1)$ we have:
$$
\lim _{x \rightarrow \infty} \frac{\alpha \log x}{x^{1-\alpha}}=0.
$$
Combining these we get:
$$
\limsup _{x \rightarrow \infty} \frac{\Theta(x)}{x} \geq \limsup _{x \rightarrow \infty} \frac{\alpha \pi(x)}{x / \log x}
$$
Once again, since our statement is true $\forall \alpha \in(0,1),$ 
$$
\limsup _{x \rightarrow \infty} \frac{\Theta(x)}{x} \geq \limsup _{x \rightarrow \infty} \frac{\pi(x)}{x / \log x}
$$
Similarly:
$$
\liminf _{x \rightarrow \infty} \frac{\Psi(x)}{x} \geq \liminf _{x \rightarrow \infty} \frac{\Theta(x)}{x}
$$
$$
\liminf _{x \rightarrow \infty} \frac{\pi(x)}{x / \log x} \geq \liminf _{x \rightarrow \infty} \frac{\Psi(x)}{x}
$$
Once again, we apply the same method:
$$
\liminf _{x \rightarrow \infty} \frac{\Theta(x)}{x} \geq \liminf _{x \rightarrow \infty} \frac{\pi(x)}{x / \log x},
$$
and have thus proven $\mathbf{2.1}$.
\end{proof}

\section{Main result}

\begin{thm}
$$
2 \log 2 \geq \limsup _{x \rightarrow \infty} \frac{\pi(x)}{x / \log x} \geq \liminf _{x \rightarrow \infty} \frac{\pi(x)}{x / \log x} \geq \log 2
$$
\begin{center}
    $\wedge \  \exists  $m, $M \in \mathbb{R}^{+}$  such that
\end{center} 
$$
M>\frac{\pi(x)}{x / \log x}>m
$$
for sufficiently large $x$.
\end{thm}
\begin{proof}
First the lower bound. Take:
$$
S(x):=\sum_{1 \leq n \leq x} \log n-2 \sum_{1 \leq n \leq x / 2} \log n.
$$

$\wedge$
$$
\begin{aligned}
\sum_{\substack{1 \leq d \leq n \\ d\mid n}} \Lambda(d)=& \sum_{\substack{p^{r}|d \\ d| n \\ p \text{ prime}}} \log p \\
=& \sum_{i=1}^{l} \sum_{r=1}^{e_{i}} \log p_{i} \quad \text { where } n=p_{1}^{e_{1}} \cdots p_{l}^{e_{l}} \\
=&\sum_{i=1}^{l} e_{i} \log p_{i} \\
=&\log n
\end{aligned}
$$

$\therefore$
$$
S(x)=\sum_{1 \leq n \leq x} \sum_{d \mid n} \Lambda(d)-2 \sum_{1 \leq n \leq x / 2} \sum_{d \mid n} \Lambda(d)
$$
Clearly $\{d, 2 d, \ldots, q d\}$ is the set of $n$ satisfying $1 \leq n \leq x$ and $d \mid n$ (we can see this easily by writing $x=r + q d$ with $0 \leq r<d$).

$\therefore$
$$
\begin{aligned}
S(x) &=\sum_{1 \leq d \leq x} \Lambda(d)\left[\frac{x}{d}\right]-2 \sum_{1 \leq d \leq x / 2} \Lambda(d)\left[\frac{x}{2 d}\right] \\
&=\sum_{1 \leq d \leq x / 2} \Lambda(d)\left(\left[\frac{x}{d}\right]-2\left[\frac{x}{2 d}\right]\right)+\sum_{(x / 2)<d \leq x} \Lambda(d)\left[\frac{x}{d}\right] \\
& \leq \sum_{1 \leq d \leq x / 2} \Lambda(d)+\sum_{(x / 2)<d \leq x} \Lambda(d) \\
&=\Psi(x).
\end{aligned}
$$
So,
$$
\frac{\Psi(x)}{x} \geq \frac{S(x)}{x}=\frac{1}{x} \sum_{1 \leq n \leq x} \log n-\frac{2}{x} \sum_{1 \leq n \leq x / 2} \log n.
$$
$\log t$ is increasing,

$\therefore$
$$
\int_{1}^{x+1} \log t \ d t \geq \sum_{1 \leq n \leq x} \log n,
$$
$$
\int_{1}^{[x]} \log t \ d t \leq \sum_{1 \leq n \leq x} \log n.
$$
Actually, assuming $x \in \mathbb{Z}^{+},$
$$
\begin{aligned}
\frac{S(x)}{x} & \geq \frac{1}{x} \int_{1}^{x} \log t \ d t-\frac{2}{x} \int_{1}^{(x / 2)+1} \log t \ d t \\
&=\frac{1}{x}(x \log x-x+1)-\frac{2}{x}\left(\frac{x+2}{2} \log \left(\frac{x+2}{2}\right)-\frac{x+2}{2}+1\right) \\
&=\log x+\frac{1}{x}-\frac{x+2}{x} \log (x+2)+\frac{x+2}{x} \log 2 \\
&>\log \left(\frac{x}{x+2}\right)-\frac{2}{x} \log (x+2)+\log 2
\end{aligned}
$$
Using $\mathbf{2.1}$, we get:
$$
\begin{aligned}
\liminf _{x \rightarrow \infty} \frac{\pi(x)}{x / \log x} &=\liminf _{x \rightarrow \infty} \frac{\Psi(x)}{x} \\
& \geq \liminf _{x \rightarrow \infty} \frac{S(x)}{x} \\
&>\lim _{x \rightarrow \infty} \log (x /(x+2))-\frac{2}{x} \log (x+2)+\log 2 \\
&=\log 2
\end{aligned}
$$

To complete the proof, we will need some auxiliary results taken from a Murty's Analytic Number Theory\cite{Murty} in the form of three lemmas:
\begin{lem}
$$\operatorname{ord}_{p}(m !)=\sum_{r \geq 1}\left[\frac{m}{p^{r}}\right], \forall m \in \mathbb{Z}^{+}, \text{ prime } p$$
\end{lem}
\begin{proof}
Fix an exponent $r$. The positive integers no larger than $m$ that are multiples of
$p^{r}$ are
$$
p^{r}, 2 p^{r}, \ldots,\left[\frac{m}{p^{r}}\right] p^{r}
$$
and those that are multiples of $p^{r+1}$ are
$$
p^{r+1}, 2 p^{r+1}, \ldots,\left[\frac{m}{p^{r+1}}\right] p^{r+1}
$$
Thus there are precisely $\left[\mathrm{m} / \mathrm{p}^{r}\right]-\left[\mathrm{m} / \mathrm{p}^{r+1}\right]$ positive integers $n \leq m$ with $\operatorname{ord}_{p}(n)=r .$

$\therefore$
$$
\begin{aligned}
\operatorname{ord}_{p}(m !) &=\sum_{n=1}^{m} \operatorname{ord}_{p}(n) \\
&=\sum_{r \geq 1} \sum_{1 \leq n \leq m \atop \text {ord}_{p}(n)=r} r \\
&=\sum_{r \geq 1} r\left(\left[m / p^{r}\right]-\left[m / p^{r+1}\right]\right) \\
&=\sum_{r \geq 1} r\left[m / p^{r}\right]-\sum_{r \geq 1} r\left[m / p^{r+1}\right] \\
&=\sum_{r \geq 1} r\left[m / p^{r}\right]-\sum_{r \geq 1}(r-1)\left[m / p^{r}\right] \\
&=\sum_{r \geq 1}\left[\frac{m}{p^{r}}\right]
\end{aligned}
$$
\end{proof}
\begin{lem}
$\forall n \in \mathbb{Z}^{+}$,
$$
\frac{2^{2 n}}{2 \sqrt{n}}<\binom{2n}{n}<\frac{2^{2 n}}{\sqrt{2 n+1}}
$$
\end{lem}
\begin{proof}
$$
\begin{aligned}
P_{n} &:=\prod_{i \leq n} \frac{(2 i-1)}{(2 i)} \\
&=\frac{(2 n) !}{2^{2 n}(n !)^{2}} \\
&=\binom{2n}{n} \frac{1}{2^{2 n}}
\end{aligned}
$$
Since:
$$
\frac{(2 i-1)(2 i+1)}{(2 i)^{2}}<1
$$
for all $i \geq 1$. 

$\therefore$
$$1>(2 n+1) P_{n}^{2},$$ 
giving the upper bound. For the lower bound:
$$
1-\frac{1}{(2 i-1)^{2}}<1
$$
$\forall i \geq 1,$ such that
$$
\begin{aligned}
1 &>\prod_{i=2}^{n}\left(1-\frac{1}{(2 i-1)^{2}}\right) \\
&=\prod_{i=2}^{n} \frac{(2 i-1)^{2}-1}{(2 i-1)^{2}} \\
&=\prod_{i=2}^{n} \frac{(2 i-2)(2 i)}{(2 i-1)^{2}} \\
&=\frac{1}{4 n P_{n}^{2}}
\end{aligned}
$$
yielding our lemma.
\end{proof}
\begin{lem}
$\forall n \in \mathbb{Z}^{+}$,
$$
\theta(n)<2 n \log 2
$$
\end{lem}
\begin{proof}

By $\mathbf{3.3}$,
$$
\begin{aligned}
\log \left(\binom{2n}{n} \frac{1}{2}\right) &=\log \left(\binom{2n}{n}\right)-\log 2 \\
&<2 n \log 2-\frac{1}{2} \log (2 n+1)-\log 2 \\
&=(2 n-1) \log 2-\frac{1}{2} \log (2 n+1)
\end{aligned}
$$
since
$$
\binom{2n}{n} \frac{1}{2}=\frac{(2 n) !}{(n !)^{2}} \frac{n}{2 n}=\frac{(2 n-1) !}{n !(n-1) !}=\binom{2n-1}{n-1}.
$$
by $\mathbf{3.2}$:
$$
\begin{aligned}
\log \left(\binom{2n}{n} \frac{1}{2}\right) &=\log \binom{2n-1}{n-1} \\
&=\sum_{p \text { prime }} \operatorname{ord}_{p}((2 n-1) !) \log p-\sum_{p \text { prime }} \operatorname{ord}_{p}((n-1) !) \log p-\sum_{p \text { prime }} \operatorname{ord}_{p}(n !) \log p \\
&=\sum_{p \text { prime }} \log p \sum_{r \geq 1}\left[(2 n-1) / p^{r}\right]-\left[(n-1) / p^{r}\right]-\left[n / p^{r}\right] \\
& \geq \sum_{\substack{p \text { prime } \\ n < p \leq 2 n-1}} \log p \\
&=\theta(2 n-1)-\theta(n)
\end{aligned}
$$
where
$$
\theta(2 n-1)-\theta(n)<(2 n-1) \log 2-\frac{1}{2} \log (2 n+1)
$$
We now proceed by induction. Proceeding from the trivialities, suppose $m>2$ and
the lemma is true for $n<m,$ $n, m \in \mathbb{Z}$. If $m$ is odd, then $m=2 n-1$ for some integer $n \geq 2$ since $m>2$.
Thus by induction,
$$
\begin{aligned}
\theta(m)=\theta(2 n-1) &<\theta(n)+(2 n-1) \log 2-\frac{1}{2} \log (2 n+1) \\
&<2 n \log 2+(2 n-1) \log 2-\frac{1}{2} \log (2 n) \\
&=(4 n-1) \log 2-\frac{1}{2} \log (2 n) \\
& \leq(4 n-2) \log 2 \quad(\text {since } n \geq 2) \\
&=2 m \log 2
\end{aligned}.
$$
If $m$ is even, then $m=2 n$ for some integer $n$ with $m>n \geq 2$ and $m$ is composite. Clearly $\theta(m)=\theta(m-1)$
and we know:
$$
\theta(m)=\theta(m-1)<2(m-1) \log 2<2 m \log 2
$$
$\mathbf{3.4}$ gives 
$$2 \log 2 \geq \lim \sup _{x \rightarrow \infty} \frac{\Theta(x)}{x} .$$
\end{proof}
The desired lower bound follows from $\mathbf{2.1}$.
\end{proof}

\section{On primes in the gaps between squares}

The following is relatively aleatory compared to the previous workings, but it is worth mentioning considering the importance of the statement.

By Hassani\cite{Hassani}, we have
$$
\left\lfloor\frac{1}{2}\left(\frac{(x+1)^{2}}{\log (x+1)}-\frac{x^{2}}{\log x}\right)-\frac{\log ^{2} x}{\log \log x}\right\rfloor \leq \pi\left((x+1)^{2}\right)-\pi\left(x^{2}\right)
$$
$$
\frac{1}{2}\left(\frac{x^{2}}{\log n}-\frac{3^{2}}{\log 3}\right)-\sum_{j=3}^{x-1} \frac{\log ^{2} x}{\log \log k}<\pi\left(x^{2}\right)-\pi\left(3^{2}\right)
$$
And thus:
$$
\sum_{j=3}^{x-1} \left \lfloor \frac{1}{2}\left(\frac{(j+1)^{2}}{\log (j+1)}-\frac{j^{2}}{\log j}\right)-\frac{\log ^{2} j}{\log \log j}\right\rfloor<\sum_{j=3}^{n-1} \pi\left((j+1)^{2}\right)-\pi\left(j^{2}\right)
$$

$\therefore$

$$
\left \lfloor\frac{1}{2}\left(\frac{\left(x+1\right)^{2}}{\log\left(x+1\right)}-\frac{x^{2}}{\log x}\right)-\frac{\left(\log x\right)^{2}}{\log\left(\log x\right)}\right \rfloor \leq \pi((x+1)^{2}) - \pi(x^{2}).
$$

\pagebreak
And by the prime number theorem, which gives us the asymptotic estimate for some
$$F(x) := \pi((x+1)^{2}) - \pi(x^{2}) \sim \frac{1}{2}\left(\frac{\left(x+1\right)^{2}}{\log\left(x+1\right)}-\frac{x^{2}}{\log x}\right)$$

We propose:
$$
\pi((x+1)^{2}) - \pi(x^{2}) \leq
\left \lfloor\frac{1}{2}\left(\frac{\left(x+1\right)^{2}}{\log\left(x+1\right)}-\frac{x^{2}}{\log x}\right) + \log^{2}x\log\log x \right \rfloor
$$

by the same method.

\end{document}